\documentclass[10pt]{amsart}
\usepackage{amsfonts}
\usepackage{amsfonts}
\usepackage{amsfonts}
\usepackage{amsfonts}
\theoremstyle{plain}
\newtheorem{Thm}{Theorem}

\newtheorem{Cor}[Thm]{Corollary}
\newtheorem{Prop}[Thm]{Proposition}
\newtheorem{Lem}[Thm]{Lemma}

\newtheorem{Def}[Thm]{Definition}
\newtheorem{Rk}[Thm]{Remark}

\begin{document}
\large

\title[uniqueness of ground states of NLS]
{Uniqueness of ground states of some coupled nonlinear Schrodinger
systems and their application}

\author{ Li Ma and Lin Zhao}

\address{Li Ma, Department of Mathematical Sciences, Tsinghua University,
 Peking 100084, P. R. China}

\email{lma@math.tsinghua.edu.cn}

\thanks{The research is partially supported by the National Natural Science
Foundation of China 10631020 and SRFDP 20060003002}

\begin{abstract}

We establish the uniqueness of ground states of some coupled
nonlinear Schrodinger systems in the whole space. We firstly use
Schwartz symmetrization to obtain the existence of ground states for
a more general case. To prove the uniqueness of ground states, we
use the radial symmetry of the ground states to transform the
systems into an ordinary differential system, and then we use the
integral forms of the system. More interestingly, as an application
of our uniqueness results, we derive a sharp
vector-valued Gagliardo-Nirenberg inequality.\\

{\bf Keywords: Schrodinger system, uniqueness of ground states,
sharp
vector-valued Gagliardo-Nirenberg inequality}\\

{\bf AMS Classification: Primary 35J.}

\end{abstract}

\maketitle
\date{7-16-2007}

\section{Introduction}

In this paper we are concerned with the uniqueness of ground states
of the coupled nonlinear Schr\"{o}dinger system:
\begin{align}\label{system}
-i\partial_t\phi_j=\Delta\phi_j+\mu_j|\phi_j|^{2p}\phi_j+\sum_{i\neq
j}\beta_{ij}|\phi_i|^{p+1}|\phi_j|^{p-1}\phi_j,
\end{align}
where $\phi_j=\phi_j(t,x)\in \mathbb{C}$, $x\in \mathbb{R}^n$,
$t>0$, $j=1,...,N$. Here $0<p<2/(n-2)^+$ (we use the convention:
$2/(n-2)^+=+\infty$ when $n=1,2$, and $(n-2)^+=n-2$ when $n\geq 3$),
$\mu_j$'s and $\beta_{ij}$'s are coupling constants subjected to
$\beta_{ij}=\beta_{ji}$.

The model (\ref{system}) has applications in many physical problems,
especially in nonlinear optics. An application of (\ref{system})
comes from \cite{AA}, the solution $\phi_j$ denotes the $j^{th}$
component of the beam in Kerr-like photo-refractive media. The
constant $\mu_j$ is for self-focusing in the $j^{th}$ component of
the beam. The coupling constant $\beta_{ij}$ is the interaction
between the $i^{th}$ and the $j^{th}$ component of the beam. We
refer to \cite{BSSSC} for more precision on the meaning of the
constants. Another application of (\ref{system}) arises in
\cite{Hioe}. When two optical waves of different frequencies
co-propagate in a medium and interact nonlinearly through the
medium, or when two polarization components of a wave interact
nonlinearly at some central frequency, the propagation equations for
the two problems can be considered together as the following $N$
coupled nonlinear Schr\"{o}dinger-like equations for the case $N=2$:
\begin{align*}
i\partial_t\phi_j+\partial_{x}^2\phi_j+\kappa_j\phi_j+
\left(\sum_{i=1}^Np_{ij}|\phi_i|^2\right)\phi_j
+\left(\sum_{i=1}^Nq_{ij}\phi_i^2\right)\bar{\phi_j}=0,
\end{align*}
where $j=1,...,N$, $\phi_j$ denotes the complex amplitude of the
$j^{th}$ electric field envelope, or the $j^{th}$ polarization
component, $p_{ij}$'s , $q_{ij}$'s and $\kappa_j$'s are parameters
characteristic of the medium and interaction. Especially, when
$\kappa_j=0$ and $q_{ij}=0$, it reduces to our model problem
(\ref{system}) where $n=1$ and $p=1$.

To obtain solitary solutions of the system (\ref{system}), we set $
\phi_j(t,x)=e^{it}u_j(x)$ ($u_j\in \mathbb{R}$) and transform the
system (\ref{system}) to steady-state $N$ coupled nonlinear
Schr\"{o}dinger equations given by
\begin{align}
u_j-\Delta u_j=\mu_j |u_j|^{2p}u_j+\sum_{i\neq
j}\beta_{ij}|u_i|^{p+1}|u_j|^{p-1}u_j,\ \ \
j=1,...,N.\label{soliton}
\end{align}
The concept of incoherent solitary solutions have attracted
considerable attentions in the last ten years, both from
experimental and theoretical point of view. The two experimental
studies \cite{MCSS} and \cite{MS} demonstrated the existence of
solitary waves made from both spatially and temporally incoherent
light. These papers were followed by a large amount of theoretical
work on incoherent solitary waves, see for example
\cite{BSSSC,Hioe,KPSV} and the references therein. The energy
functional of (\ref{soliton}) is
\begin{align*}
\mathcal{E}({\bf
u}):=&\frac{1}{2}\sum_{i=1}^N\int_{\mathbb{R}^n}\left(|\nabla
u_i|^2+u_i^2\right)-\frac{1}{2p+2}\sum_{i=1}^N\int_{\mathbb{R}^n}
\mu_iu_i^{2p+2}\\
&-\frac{1}{2p+2}\sum_{i,j=1}^N\int_{\mathbb{R}^n}\beta_{ij}|u_i|^{p+1}|u_j|^{p+1}.
\end{align*}
This functional is well defined if $u_i\in H^1(\mathbb{R}^n)$, by
virtue of the embedding $H^1(\mathbb{R}^n)\hookrightarrow
L^{2p+2}(\mathbb{R}^n)$ with $0<p<2/(n-2)^{+}$. We will always
consider solitary waves with finite energy, and will be particularly
interested in the least energy nontrivial solutions of
(\ref{soliton}), which are named ground states in Physics.

Let's recall some previous work about the ground states of
(\ref{soliton}) related to our research in this paper. In order to
simplify the presentation, we shall concentrate on the system of two
equations:
\begin{align}
\label{model}\left\{\begin{array}{ll} u_1-\Delta u_1=
\mu_1|u_1|^{2p}u_1+\beta|u_2|^{p+1}|u_1|^{p-1}u_1,\\
u_2-\Delta u_2=\mu_2|u_2|^{2p}u_2+\beta|u_1|^{p+1}|u_2|^{p-1}u_2.
\end{array}
\right.
\end{align}
A solution ${\bf u}=(u_1,u_2)$ of (\ref{model}) is called nontrivial
if $u_1\not\equiv 0$ and $u_2\not\equiv 0$ simultaneously. The
nontrivial weak solutions of (\ref{model}) are equivalent to the
nontrivial critical points of the energy functional
\begin{align*}
\mathcal{E}({\bf u})=&\frac{1}{2}\int_{\mathbb{R}^n}\left(|\nabla
u_1|^2+u_1^2+|\nabla
u_2|^2+u_2^2\right)\\
&-\frac{1}{2p+2}\int_{\mathbb{R}^n}\left(\mu_1u_1^{2p+2}+
2\beta|u_1|^{p+1}|u_2|^{p+1}+\mu_2u_2^{2p+2}\right)
\end{align*}
in the Sobolev space $H:=H^1(\mathbb{R}^n)\times H^1(\mathbb{R}^n)$.
Notice that any nontrivial solution of (\ref{model}) has to belong
to the Nehari manifold
\begin{align*}
\mathcal{N}:=\{&{\bf u}\in H, u_1\not\equiv0, u_2\not\equiv0;\ \ \
\int_{\mathbb{R}^n}\left(|\nabla u_1|^2+u_1^2+|\nabla
u_2|^2+u_2^2\right)\\
&=\int_{\mathbb{R}^n}\left(\mu_1u_1^{2p+2}+
2\beta|u_1|^{p+1}|u_2|^{p+1}+\mu_2u_2^{2p+2}\right)\}.
\end{align*}

\begin{Def}
The nonnegative minima of the minimization problem
\begin{align}
c:=\inf_{{\bf u}\in\mathcal{N}}\mathcal{E}({\bf u})\label{minimal}
\end{align}
is called the ground state of (\ref{model}).
\end{Def}

In the case of a single nonlinear Schr\"{o}dinger equation, the
ground state exists \cite{BL} and was proved to be the positive
solution of
\begin{align}
\Delta u-u+u^{2p+1}=0.\label{single}
\end{align}
The positive solution of (\ref{single}) is radial symmetric about
some fixed point \cite{GNN} and is unique in the sense of moduling
translations \cite{Kwong}. We denote it by $\omega$ hereafter.

Quite differently from the case of a single equation, the existence
of ground states solutions with multi-components of the system
(\ref{soliton}) is much more complicated than the single case and
was studied quite well when $\mu_j>0$ in the series of the papers
\cite{AC,LW,MMP,Sirakov}. Roughly speaking, they proved that there
always exist ranges of positive parameters $u_j$'s, $\beta_{ij}$'s
in (\ref{soliton}), for which this system has a least energy
solution, and ranges of positive parameters for which the energy
functional can't be minimized on the Nehari manifold where the
eventual solutions lie. Readers can consult these papers for further
details.

However, the uniqueness of positive solutions of the system
(\ref{soliton}) is a widely open problem, and to our knowledge no
results have been already known in this direction. In our present
paper, we discuss the ground state of (\ref{soliton}) in the case
$\mu_j\leq0$, which has not been considered before, and we will
prove that in this case the ground state is unique. The uniqueness
of ground states when $\mu_j>0$ remains open. To be precise, our
result reads as follows.

\begin{Thm}\label{main}
Consider the steady-state two coupled nonlinear Schr\"{o}dinger
equations in $\mathbb{R}^n$
\begin{align}
\label{thm}\left\{\begin{array}{ll} u_1-\Delta u_1=
\mu_1|u_1|^{2p}u_1+\beta_1|u_2|^{p+1}|u_1|^{p-1}u_1,\\
u_2-\Delta u_2=\mu_2|u_2|^{2p}u_2+\beta_2|u_1|^{p+1}|u_2|^{p-1}u_2,
\end{array}
\right.
\end{align}
in which $0<p<2/(n-2)^+$. Assume that
\begin{align}
\mu_1,\mu_2\leq0,\ \ \ \beta_1,\beta_2>0, \  \  \
\mu_1\beta_1^p=\mu_2\beta_2^p,\label{assumption1}
\end{align}
and
\begin{align}
\mu_1+\frac{\beta_2^{(p+1)/2}}{\beta_1^{(p-1)/2}}>0 \ \ \
\textrm{or} \ \ \
\mu_2+\frac{\beta_1^{(p+1)/2}}{\beta_2^{(p-1)/2}}>0.\label{assumption2}
\end{align}
Then the ground state of (\ref{thm}) exists and is unique up to
translations. Moreover, the ground state can be determined
explicitly by
\begin{align}
\label{solution}\left\{\begin{array}{ll}
u_1=\left(\mu_1+\frac{\beta_2^{(p+1)/2}}{\beta_1^{(p-1)/2}}\right)^{-1/2p}\omega,\\
u_2=\left(\mu_2+\frac{\beta_1^{(p+1)/2}}{\beta_2^{(p-1)/2}}\right)^{-1/2p}\omega,
\end{array}
\right.
\end{align}
where $\omega$ is the unique positive solution of (\ref{single}).
\end{Thm}

As far as we know, there are only two results about the uniqueness
of positive solutions to stationary Schrodinger systems. One is in
\cite{LiMa}. The other one is in \cite{Ma}, where the radial
symmetry and uniqueness results have been obtained for the
non-negative solutions to the schrodinger system
$$
(I-\Delta)u=v^{p},\\
(I-\Delta)v=u^{q}.
$$
One may see \cite{ChenLi} for related uniqueness result.

 We now give
some remarks about the conditions (\ref{assumption1}) and
(\ref{assumption2}).

\begin{Rk}
The condition (\ref{assumption1}) implies that
$$
\left(\mu_2+\frac{\beta_1^{(p+1)/2}}{\beta_2^{(p-1)/2}}\right)=
\frac{\beta_1^p}{\beta_2^p}\left(\mu_1+\frac{\beta_2^{(p+1)/2}}{\beta_1^{(p-1)/2}}\right),
$$
and hence in (\ref{assumption2})
$$
\mu_1+\frac{\beta_2^{(p+1)/2}}{\beta_1^{(p-1)/2}}>0\Leftrightarrow
\mu_2+\frac{\beta_1^{(p+1)/2}}{\beta_2^{(p-1)/2}}>0.
$$
\end{Rk}

\begin{Rk}\label{Rk}
It's easy to check that the special case
\begin{align}
\mu_1=\mu_2=\mu\leq 0,\ \ \ \beta_1=\beta_2=\beta\ \ \ \textrm{and}
\ \ \ \mu+\beta>0\label{standard}
\end{align}
satisfies (\ref{assumption1}) and (\ref{assumption2}). Conversely,
any other constants satisfying (\ref{assumption1}) and
(\ref{assumption2}) can be transformed to the case (\ref{standard})
by scaling. In fact, set
$$
w_1(x):=a_1u_1(x),\ \ \ w_2(x):=a_2u_2(x),
$$
where $a_1,a_2>0$ are the scaling constants. Then the equations
(\ref{thm}) can be written as
\begin{align*}
\left\{\begin{array}{ll} w_1-\Delta w_1=
\frac{\mu_1}{a_1^{2p}}w_1^{2p+1}+\frac{\beta_1}{a_1^{p-1}a_2^{p+1}}w_1^pw_2^{p+1},\\
w_2-\Delta w_2=\frac{\mu_2}{a_2^{2p}}
w_2^{2p+1}+\frac{\beta_2}{a_2^{p-1}a_1^{p+1}} w_2^pw_1^{p+1}.
\end{array}
\right.
\end{align*}
The condition (\ref{assumption1}) guarantees the existence of
$a_1,a_2>0$ such that
$$
\frac{\mu_1}{a_1^{2p}}=\frac{\mu_2}{a_2^{2p}}:=\mu,\ \ \
\frac{\beta_1}{a_1^{p-1}a_2^{p+1}}=\frac{\beta_2}{a_2^{p-1}a_1^{p+1}}:=\beta.
$$
Indeed, we can choose without loss of generality that
$$
a_1=1,\ \ \ a_2=(\frac{\beta_1}{\beta_2})^{1/2}
$$
and thus
$$
\mu=\mu_1,\ \ \ \beta=\frac{\beta_2^{(p+1)/2}}{\beta_1^{(p-1)/2}}.
$$
The condition (\ref{assumption2}) is just that $\mu+\beta>0$.
\end{Rk}

\begin{Rk}
The relation $\mu+\beta>0$ in (\ref{standard}) plays a crucial role
to ensure that the Nehari manifold $\mathcal{N}\neq\emptyset$.
\end{Rk}

As is well known, the sharp Gagliardo-Nirenberg inequality plays
extremely important roles in the quantitative analysis of blow-up
solutions of the single Schr\"{o}dinger equation. A large amount of
work relies heavily on the sharp constant in the Gagliardo-Nirenberg
inequality. We shall only quote here
\cite{Merle1,Merle2,Tsutsumi,Weinstein1,Weinstein2} where a
comprehensive list of references on this subject can be found. As an
application of Theorem \ref{main}, we can obtain a sharp
vector-valued Gagliardo-Nirenberg inequality. To our experience, the
sharp vector-valued Gagliardo-Nirenberg inequality we obtain here
would play some non-negligible roles in further studies of
Schr\"{o}dinger systems.

\begin{Cor}\label{GN}
Let $0<p<2/(n-2)^+$ and $\mathcal{K}_{n,p}$ be the sharp constant in
the single valued Gagliardo-Nirenberg inequality, that is,
$$
\|u\|_{2p+2}^{2p+2}\leq \mathcal{K}_{n,p}\|u\|_2^{2p+2-np}\|\nabla
u\|_2^{np},\ \ \ \forall\ u\in H^1(\mathbb{R}^n).
$$
Assume the constants $\mu$, $\beta$ satisfy
$$
\mu\leq 0,\ \ \ \textrm{and}\ \ \ \mu+\beta>0.
$$
Then we have the two vector-valued Gagliardo-Nirenberg inequality
below. That is, $\forall\ u_1,u_2\in H^1(\mathbb{R}^n)$,
\begin{align}
&\quad\mu\|u_1\|_{2p+2}^{2p+2}+2\beta\|u_1u_2\|_{p+1}^{p+1}+\mu\|u_2\|_{2p+2}^{2p+2}\label{GNi}\\
&\leq
\mathcal{K}_{n,p,\mu,\beta}\left(\|u_1\|_2^2+\|u_2\|_2^2\right)^{p+1-np/2}
\left(\|\nabla u_1\|_2^2+\|\nabla u_2\|_2^2\right)^{np/2},\nonumber
\end{align}
in which the sharp constant $\mathcal{K}_{n,p,\mu,\beta}$ is
determined by
$$
\mathcal{K}_{n,p,\mu,\beta}=\frac{(\mu+\beta)}{2^p}\mathcal{K}_{n,p}.
$$
\end{Cor}

\begin{Rk}
If one uses H{\"{o}}lder inequality directly, one can only get a
vector-valued Gagliardo-Nirenberg inequality like
\begin{align*}
&\quad
\mu\|u_1\|_{2p+2}^{2p+2}+2\beta\|u_1u_2\|_{p+1}^{p+1}+\mu\|u_2\|_{2p+2}^{2p+2}\\
&\leq(\mu+\beta)\left(\|u_1\|^{2p+2}_{2p+2}+\|u_2\|_{2p+2}^{2p+2}\right)\\
&\leq(\mu+\beta)\left(\|u_1\|^2_{2p+2}+\|u_2\|_{2p+2}^2\right)^{p+1}\\
&\leq(\mu+\beta)\mathcal{K}_{n,p}\left(\sum_{j=1}^2(\|u_j\|_2^2)^{\frac{p+1-np/2}{p+1}}(\|\nabla
u_j\|_2^2)^{\frac{np/2}{p+1}}\right)^{p+1}\\
&\leq(\mu+\beta)\mathcal{K}_{n,p}\left((\|u_1\|_2^2+\|u_2\|_2^2)^{\frac{p+1-np/2}{p+1}}
(\|\nabla u_1\|_2^2+\|\nabla
u_2\|_2^2)^{\frac{np/2}{p+1}}\right)^{p+1}\\
&=(\mu+\beta)\mathcal{K}_{n,p}\left(\|u_1\|_2^2+\|u_2\|_2^2\right)^{p+1-np/2}
\left(\|\nabla u_1\|_2^2+\|\nabla u_2\|_2^2\right)^{np/2},
\end{align*}
in which the constant $(\mu+\beta)\mathcal{K}_{n,p}$ is in strong
contrast with the sharp constant $(\mu+\beta)\mathcal{K}_{n,p}/2^p$.
In fact, the sharp constant relies heavily on the explicit
expressions of ground states.
\end{Rk}

\begin{Rk}
Somewhat surprisingly, our arguments to prove Theorem \ref{main} and
corollary \ref{GN} can't be generalized to the $N$ coupled
Schr\"{o}dinger system with $N\geq 3$. So we have to leave the case
$N\geq 3$ as an open problem. For the scalar Gagliardo-Nirenberg
inequality in its general form, one may see E.Hebey's book
\cite{He}.
\end{Rk}

To prove Theorem \ref{main}, we only deal with the standard case
(\ref{standard}) of (\ref{thm}), as what we have explained in Remark
\ref{Rk}. In section 2, we use Schwartz symmetrization to prove that
the minimization problem (\ref{minimal}) can be achieved by a
positive solution of the system (\ref{thm}), which indicates the
existence of ground states in a more general case. In section 3, we
transform (\ref{thm}) to a system of ordinary differential equations
(ODE) by virtue of the radial symmetry of positive solutions of
(\ref{thm}). Then by the comparison technique of ODE, we arrive at
the uniqueness of positive solutions of (\ref{thm}). Since all the
ground sates must be positive solutions of (\ref{thm}), we conclude
that the ground state is unique. In section 4, we prove the sharp
vector-valued Gagliardo-Nirenberg inequality (Corollary \ref{GN}) in
detail.

\section{Existence of ground states}

This section is devoted to the proof of the existence of ground
states of (\ref{thm}) in the case
\begin{align}\mu_1,\mu_2\leq 0\ \
\textrm{and}\ \ \ \mu_1+\beta>0,\ \ \ \mu_2+\beta>0.\label{sect1}
\end{align}
The existence of ground states of (\ref{thm}) when $\mu_1,\mu_2>0$
has been extensively studied in the papers \cite{AC,LW,MMP,Sirakov}
using the method of Schwartz symmetrization. We declare that this
symmetrization method still works for the case (\ref{sect1}) under
our consideration. Since our proof would have many details different
from the ones in the preceding papers, we will give our proof
thoroughly for the purpose of completeness. We point out that our
proof, which combines the analysis in \cite{AC} and \cite{LW}, could
be seen as a simplified version of their arguments.

We have the following proposition, which asserts that all the
critical points of the minimization problem (\ref{minimal}) must be
weak solutions of (\ref{thm}) in $H$.

\begin{Prop}\label{prop}
If the minimization problem (\ref{minimal}) is attained by a coupled
${\bf u}\in\mathcal{N}$,  then ${\bf u}$ is a solution of
(\ref{thm}).
\end{Prop}

\begin{proof}
The proof of Proposition \ref{prop} is similar to the one in
\cite{AC}. Let
\begin{align*}
\mathcal{G}({\bf u}):=&\int_{\mathbb{R}^n}\left(|\nabla
u_1|^2+u_1^2+|\nabla
u_2|^2+u_2^2\right)\\
&-\int_{\mathbb{R}^n}\left(\mu_1 u_1^{2p+2}+
2\beta|u_1|^{p+1}|u_2|^{p+1}+\mu_2 u_2^{2p+2}\right).
\end{align*}
We have for each $\psi=(\psi_1,\psi_2)\in H$ that
\begin{align*}
<\nabla \mathcal{E}({\bf u}),\psi>
=&\sum_{i=1}^2\int_{\mathbb{R}^n}\left(\nabla
u_i\cdot\nabla\psi_i+u_i\psi_i-\mu_iu_i^{2p+1}\psi_i\right)\\
&-\sum_{i=1}^2\int_{\mathbb{R}^n}\beta|u_i|^{p-1}|u_j|^{p+1}u_i\psi_i,\
\ \ j\neq i,
\end{align*}
\begin{align*}
<\nabla \mathcal{G}({\bf u}),\psi>
=&2\sum_{i=1}^2\int_{\mathbb{R}^n}\left(\nabla
u_i\cdot\nabla\psi_i+u_i\psi_i-(p+1)\mu_iu_i^{2p+1}\psi_i\right)\\
&-2(p+1)\sum_{i=1}^2\int_{\mathbb{R}^n}\beta
|u_i|^{p-1}|u_j|^{p+1}u_i\psi_i.\ \ \ j\neq i.
\end{align*}
Suppose that ${\bf u}=(u_1,u_2)\in\mathcal{N}$ is a minimizer for
$\mathcal{E}$ restricted on $\mathcal{N}$, then the standard
minimization theory yields an Euler-Lagrange multiplier
$L\in\mathbb{R}$ such that
$$
\nabla\mathcal{E}({\bf u})+L\nabla\mathcal{G}({\bf u})=0.
$$
Setting $\mathcal{G}({\bf u})=<\nabla\mathcal{E}({\bf u}),{\bf
u}>=0$ in the expression $<\nabla\mathcal{E}({\bf
u})+L\nabla\mathcal{G}({\bf u}),{\bf u}>=0$, we obtain that
$$
L\int_{\mathbb{R}^n}\left(|\nabla u_1|^2+u_1^2+|\nabla
u_2|^2+u_2^2\right)=0,
$$
which implies that $L=0$, thanks to ${\bf u}\not\equiv 0$.
\end{proof}

Next, we use Schwartz symmetrization to prove that the minimum $c$
in (\ref{minimal}) can be achieved by a positive solution of
(\ref{thm}) as in \cite{LW}. The following lemma \cite{Lieb} is at
the heart of our argument.

\begin{Lem}\label{lem}
Let $u^*$ be the Schwartz symmetric function associated to $u$,
namely the radially symmetric, radially non-increasing function,
equi-measurable with $u$. There hold for $1\leq p<\infty$ that
$$
\int_{\mathbb{R}^n}|\nabla u^*|^2\leq\int_{\mathbb{R}^n}|\nabla
u|^2,\ \ \ \forall\ u\in H^1(\mathbb{R}^n),\ \ \ u\geq0;
$$
$$
\int_{\mathbb{R}^n}|u^*|^p=\int_{\mathbb{R}^n}|u|^p,\ \ \ \forall\
u\in L^p({R}^n), \ \ \ u\geq 0;
$$
$$
\int_{\mathbb{R}^n}(u^*)^{p}(v^*)^{p}\geq
\int_{\mathbb{R}^n}u^{p}v^{p},\ \ \ \forall\ u,v\in
L^{2p}(\mathbb{R}^n),\ \ \ u,v\geq 0.
$$
\end{Lem}

After these preparations, we now state and prove the main result in
this section.

\begin{Thm}\label{exi}
Assume (\ref{sect1}). Then the ground states of (\ref{thm}) exist
and are positive solutions of (\ref{thm}).
\end{Thm}

\begin{proof}
With the help of Proposition \ref{prop}, we see it remains only to
verify that the minimum of $c$ in (\ref{minimal}) can be attained by
a pair of positive functions in $\mathcal{N}$. Define
\begin{align*}
\overline{\mathcal{N}}:=\{&{\bf u}\in H, u_1\not\equiv0,
u_2\not\equiv0;\ \ \ \int_{\mathbb{R}^n}\left(|\nabla
u_1|^2+u_1^2+|\nabla
u_2|^2+u_2^2\right)\\
&\leq\int_{\mathbb{R}^n}\left(\mu_1u_1^{2p+2}+
2\beta|u_1|^{p+1}|u_2|^{p+1}+\mu_2u_2^{2p+2}\right)\},
\end{align*}
and
$$
\overline{c}:=\inf_{{\bf u}\in
\overline{\mathcal{N}}}\mathcal{E}({\bf u}).
$$
The conditions $\mu_1+\beta>0$, $\mu_2+\beta>0$ ensure that
$\mathcal{N}\neq\emptyset$, $\overline{\mathcal{N}}\neq\emptyset$.
It's obviously that $\overline{c}\leq c$.

Step 1. Noting that
$$
\mathcal{E}({\bf
u})\geq\frac{p}{2p+2}\int_{\mathbb{R}^n}\left(|\nabla
u_1|^2+u_1^2+|\nabla u_2|^2+u_2^2\right)>0, \ \ \ \forall \ {\bf
u}\in \overline{\mathcal{N}},
$$
the definition of $\overline{c}$ makes sense. Since
$\mathcal{E}(u_1,u_2)=\mathcal{E}(|u_1|,|u_2|)$, we can take a
nonnegative minimizing sequence $\{{\bf u}_k\}$ of $\overline{c}$.
We use $C$ to denote various constants independent of ${\bf u_k}$.
By Sobolev embedding and $\mu_1,\mu_2\leq 0$, it follows that for
all ${\bf u}_k\in\overline{\mathcal{N}}$ that
\begin{align*}
&\quad\|u_{k,1}\|_{2p+2}\|u_{k,2}\|_{2p+2}\\
&\leq \frac{1}{2}(\|u_{k,1}\|^2_{2p+2}+\|u_{k,2}\|^2_{2p+2})\\
&\leq C\int_{\mathbb{R}^n}\left(|\nabla u_{k,1}|^2+u_{k,1}^2+|\nabla
u_{k,2}|^2+u_{k,2}^2\right)\\
&\leq C\int_{\mathbb{R}^n}\left(\mu_1u_{k,1}^{2p+2}+
2\beta|u_{k,1}|^{p+1}|u_{k,2}|^{p+1}+\mu_2u_{k,2}^{2p+2}\right)\\
&\leq C\beta\|u_{k,1}\|^{p+1}_{2p+2}\|u_{k,2}\|_{2p+2}^{p+1},
\end{align*}
which implies that $\|u_{k,1}\|_{2p+2}\|u_{k,2}\|^{2p+2}\geq C>0$.
Let ${\bf u}_k^*=(u_{k,1}^*, u_{k,2}^*)$ be the Schwartz
symmetrization of ${\bf u}_k$. By lemma \ref{lem} one checks easily
that
\begin{align}
{\bf u}_k^*\in \overline{\mathcal{N}},\ \ \
\|u_{k,1}^*\|_{2p+2}\|u_{k,2}^*\|_{2p+2}\geq C>0,\label{nonzero}
\end{align}
and ${\bf u}_k^*$ is also a minimizing sequence of $\overline{c}$.
By the well-known compact embedding from radial symmetric functions
in $H^1(\mathbb{R}^n)$ to $L^{2p+2}(\mathbb{R}^n)$
\cite{Weinstein1}, one can assume that ${\bf u}^*_k\rightarrow {\bf
u}^*$ in $L^{2p+2}(\mathbb{R}^n)\times L^{2p+2}(\mathbb{R}^n)$. By
Fatou's lemma, ${\bf u}^*\in\overline{\mathcal{N}}$, and
$$
\overline{c}=\mathcal{E}({\bf u}^*).
$$
Moreover, from (\ref{nonzero}), we deduce that $u^*_{1}\not\equiv 0,
u^*_2\not\equiv0$.

Step 2. We claim that
\begin{align*}
&\quad\int_{\mathbb{R}^n}\left(|\nabla u_1^*|^2+u_1^{*2}+|\nabla
u_2^*|^2+u_2^{*2}\right)\\
&=\int_{\mathbb{R}^n}\left(\mu_1u_1^{*2p+2}+
2\beta|u_1^*|^{p+1}|u_2^*|^{p+1}+\mu_2u_2^{*2p+2}\right).
\end{align*}
Suppose not, we have
\begin{align*}
&\quad\int_{\mathbb{R}^n}\left(|\nabla u_1^*|^2+u_1^{*2}+|\nabla
u_2^*|^2+u_2^{*2}\right)\\
&<\int_{\mathbb{R}^n}\left(\mu_1u_1^{*2p+2}+
2\beta|u_1^*|^{p+1}|u_2^*|^{p+1}+\mu_2u_2^{*2p+2}\right).
\end{align*}
Then ${\bf u}^*$ belongs to the interior of
$\overline{\mathcal{N}}$, that is, ${\bf u}^*$ is an interior
critical point of $\mathcal{E}({\bf u})$, and this leads to
$$
\nabla\mathcal{E}({\bf u}^*)=0,
$$
which implies that ${\bf u}^*$ is a weak solution of (\ref{thm}).
Multiplying (\ref{thm}) by ${\bf u}^*$ and integrating over
$\mathbb{R}^n$ by parts, we have
\begin{align*}
&\quad\int_{\mathbb{R}^n}\left(|\nabla u_1^*|^2+u_1^{*2}+|\nabla
u_2^*|^2+u_2^{*2}\right)\\
&=\int_{\mathbb{R}^n}\left(\mu_1u_1^{*2p+2}+
2\beta|u_1^*|^{p+1}|u_2^*|^{p+1}+\mu_2u_2^{*2p+2}\right),
\end{align*}
which is a contradiction.

Step 3. From Step 1 and Step 2, we have that
$$
c=\overline{c}=\mathcal{E}({\bf u}^*).
$$
By Proposition \ref{prop}, ${\bf u}^*$ is a nonnegative solution of
(\ref{thm}) such that $u^*_{1}\not\equiv 0, u^*_2\not\equiv0$. The
maximum principle applied to each single equation in (\ref{thm})
suggests that $u^*_{1}>0, u^*_2>0$ and the proof of the existence of
ground states of (\ref{thm}) is finished.

Step 4. We assert that all the ground states must be positive
solutions of (\ref{thm}). In fact, Proposition \ref{prop}
demonstrates that all the ground states are nonnegative solutions of
(\ref{thm}) and each component of the solutions is nonzero. By the
strong maximum principle, these solutions must be strictly positive.
The proof of Theorem \ref{exi} is complete.
\end{proof}

\section{Uniqueness of ground states}

We are now in position to prove the uniqueness of the ground states
in the case
\begin{align*}
\mu_1=\mu_2=\mu\leq 0,\ \ \ \textrm{and}\ \ \ \mu+\beta>0.
\end{align*}
The positive weak solutions of (\ref{thm}) in $H$ when $\beta>0$
were proved to be regular enough, be radial symmetric up to
translations, and decay to zero exponentially as $|x|\rightarrow
+\infty$ in \cite{Busca,Ma}. If we denote
$$
u_1(x)=u_1(|x|)=u_1(r),\ \ \ u_2(x)=u_2(|x|)=u_2(r),
$$
we are then led to the following ODE system:
\begin{align}
\label{ODE}\left\{\begin{array}{ll}
-(r^{n-1}u_1')'+r^{n-1}u_1=\mu r^{n-1}u_1^{2p+1}+\beta r^{n-1}u_1^pu_2^{p+1},\\
-(r^{n-1}u_2')'+r^{n-1}u_2=\mu r^{n-1}u_2^{2p+1}+\beta r^{n-1}u_2^pu_1^{p+1}.\\
\end{array}
\right.
\end{align}
By the radial symmetry again we have that
$$
u_1'(0)=u_2'(0)=0.
$$
Integrating (\ref{ODE}) from $0$ to $r$ we have
\begin{align}
\label{int1}\left\{\begin{array}{ll}
u_1'(r)=r^{1-n}\int_0^r t^{n-1}u_1-\mu r^{1-n}\int_0^rt^{n-1}u_1^{2p+1}-\beta r^{1-n}\int_0^r t^{n-1}u_1^pu_2^{p+1},\\
u_2'(r)=r^{1-n}\int_0^r t^{n-1}u_2-\mu r^{1-n}\int_0^rt^{n-1}u_2^{2p+1}-\beta r^{1-n}\int_0^r t^{n-1}u_2^pu_1^{p+1}.\\
\end{array}
\right.
\end{align}
Integrating once again from $0$ to $r$ we achieve
\begin{align}
\label{int2}\left\{\begin{array}{ll}
u_1(r)&=u_1(0)+\int_0^rt^{1-n}\int_0^t s^{n-1}u_1(s)-\mu\int_0^rt^{1-n}\int_0^ts^{n-1}u_1^{2p+1}(s)\\
&\quad-\beta\int_0^r t^{1-n}\int_0^ts^{n-1}u_1^pu_2^{p+1}(s),\\
u_2(r)&=u_2(0)+\int_0^rt^{1-n}\int_0^t s^{n-1}u_2(s)-\mu\int_0^rt^{1-n}\int_0^ts^{n-1}u_2^{2p+1}(s)\\
&\quad-\beta\int_0^r t^{1-n}\int_0^ts^{n-1}u_2^pu_1^{p+1}(s).\\
\end{array}
\right.
\end{align}

We claim that $u_1(0)=u_2(0)$. If else, suppose that $u_1(0)>u_2(0)$
for example, and define
$$
R_0:=\sup_{R>0}\{R;\ \  \forall\ r\in(0,R),\  \  u_1(r)>u_2(r)\}.
$$
We indicate that $R_0=+\infty$. Otherwise, by continuity we have
\begin{align}
u_1(R_0)=u_2(R_0).\label{R}
\end{align}
However, from (\ref{int2}), and the facts that for all $s\in
(0,R_0)$
\begin{align*}
\left\{\begin{array}{ll} u_1(s)-u_2(s)>0,\\
-\mu\left(u_1^{2p+1}(s)-u_2^{2p+1}(s)\right)>0,\\
-\beta\left(u_1^pu_2^{p+1}(s)-u_2^pu_1^{p+1}(s)\right)>0,
\end{array}
\right.
\end{align*}
we have
\begin{align*}
u_1(R_0)-u_2(R_0)&=(u_1(0)-u_2(0))+\int_0^{R_0}t^{1-n}\int_0^t
s^{n-1}(u_1(s)-u_2(s))\\
&\quad-\mu\int_0^{R_0}t^{1-n}\int_0^ts^{n-1}(u_1^{2p+1}(s)-u_2^{2p+1}(s))\\
&\quad-\beta\int_0^{R_0}
t^{1-n}\int_0^ts^{n-1}(u_1^pu_2^{p+1}(s)-u_2^pu_1^{p+1}(s))>0,
\end{align*}
which is a contradiction with (\ref{R}). A further fact about $u_1$
and $u_2$ is that $(u_1-u_2)(r)$ is nondecreasing as $r$ goes into
infinity. Indeed, from (\ref{int1}) we have
\begin{align*}
u_1'(r)-u_2'(r)&=r^{1-n}\int_0^r t^{n-1}(u_1-u_2)-\mu
r^{1-n}\int_0^rt^{n-1}(u_1^{2p+1}-u_2^{2p+1})\\
&\quad-\beta r^{1-n}\int_0^r t^{n-1}(u_1^pu_2^{p+1}-u_2^pu_1^{p+1})>
0,
\end{align*}
where the inequality follows from $u_1>u_2$. Thus we have
$$
\liminf_{r\rightarrow+\infty}(u_1-u_2)(r)\geq u_1(0)-u_2(0)>0,
$$
which contradicts with the fact that $u_1,u_2\rightarrow 0$ as
$r\rightarrow +\infty$. Similarly, one can show that $u_1(0)<u_2(0)$
is also impossible.

Now we have $u_1(0)=u_2(0)$ and $u_1'(0)=u_2'(0)=0$. We deduce from
the standard uniqueness theory of the Cauchy problem of the ODE
system that
$$
u_1=u_2=u,
$$
where $u$ is the positive solution of
$$
\Delta u-u+(\mu+\beta)u^{2p+1}=0.
$$
Since the above equation has only one positive solution \cite{Kwong}
up to translations given by
$$
u=(\mu+\beta)^{-1/2p}\omega,
$$
we arrive at the uniqueness of positive solutions of (\ref{thm}),
and the proof of Theorem \ref{main} is finished.

\section{Sharp vector-valued Gagliardo-Nirenberg inequality}

In this section, we derive the sharp vector-valued
Gagliardo-Nirenberg inequality (Corollary \ref{GN}) as an
application of our uniqueness result of ground states and the method
of \cite{Weinstein1} (see also \cite{CW}). We define the following
manifold
$$
\mathcal{M}:=\{u_1,u_2\in H^1(\mathbb{R}^n);\ \ \
\mu\|u_1\|_{2p+2}^{2p+2}+2\beta\|u_1u_2\|_{p+1}^{p+1}+\mu\|u_2\|_{2p+2}^{2p+2}>0
\},
$$
and consider the minimization problem
\begin{align*}
\alpha:=\inf_{{\bf u}\in\mathcal{M}}\mathcal{J}({\bf u}),
\end{align*}
where
\begin{align*}
\mathcal{J}({\bf
u})=\frac{(\|u_1\|_2^2+\|u_2\|_2^2)^{p+1-np/2}(\|\nabla
u_1\|_2^2+\|\nabla
u_2\|_2^2)^{np/2}}{\mu\|u_1\|_{2p+2}^{2p+2}+2\beta\|u_1u_2\|_{p+1}^{p+1}+\mu\|u_2\|_{2p+2}^{2p+2}}.
\end{align*}
It's obvious that the sharp constant in the vector-valued
Gagliardo-Nirenberg inequality (\ref{GNi}) is
$$
\mathcal{K}_{n,p,\mu,\beta}=\frac{1}{\alpha}.
$$

Applying the same method exactly as in \cite{Weinstein1}, we assert
that the minimum of $\alpha$ can be achieved by a pair of positive
solutions $u_1^*, u_2^*$ of
\begin{align}
\label{GN1}\left\{\begin{array}{ll} u^*_1-\Delta u^*_1=
\mu{u_1^*}^{2p+1}+\beta {u_2^*}^{p+1}{u_1^*}^{p},\\
u^*_2-\Delta u^*_2= \mu{u_2^*}^{2p+1}+\beta
{u_1^*}^{p+1}{u_2^*}^{p},
\end{array}
\right.
\end{align}
Multiplying (\ref{GN1}) by ${\bf u}^*$ and integrating by parts over
$\mathbb{R}^n$, we have
\begin{align*}
\|\nabla
u_j^*\|_2^2+\|u_j^*\|_2^2=\mu\|u_j^*\|_{2p+2}^{2p+2}+\beta\|u_1^*u_2^*\|^{p+1}_{p+1},
\end{align*}
which yields
\begin{align}
\sum_{j=1}^2\|\nabla
u_j^*\|_2^2+\sum_{j=1}^2\|u_j^*\|_2^2=\mu\sum_{j=1}^2\|u_j^*\|_{2p+2}^{2p+2}+2\beta\|u_1^*u_2^*\|^{p+1}_{p+1}.\label{1}
\end{align}
Moreover, the Pohozaev identity for (\ref{GN1}) reads
\begin{align}
&\quad\frac{n-2}{2}\sum_{j=1}^2\|\nabla
u_j^*\|_2^2+\frac{n}{2}\sum_{j=1}^2\|u_j^*\|_2^2\label{2}\\
&=\frac{n}{2p+2}\left(\mu\sum_{j=1}^2\|u_j^*\|^{2p+2}_{2p+2}+2\beta\|u_1^*u_2^*\|^{p+1}_{p+1}\right).\nonumber
\end{align}
From (\ref{1}) and (\ref{2}), we get that
\begin{align*}
\left\{\begin{array}{ll}
\left(\mu\sum_{j=1}^2\|u_j^*\|^{2p+2}_{2p+2}+2\beta\|u_1^*u_2^*\|^{p+1}_{p+1}\right)
=\frac{2p+2}{2p+2-np}\sum_{j=1}^
2\|u_j^*\|_2^2,\\
\sum_{j=1}^2\|\nabla u_j^*\|_2^2=\frac{np}{2p+2-np}
\sum_{j=1}^2\|u_j^*\|_2^2,
\end{array}
\right.
\end{align*}
which gives
$$
\mathcal{J}({\bf
u^*})=\frac{(np)^{np/2}(2p+2-np)^{1-np/2}}{2(p+1)}\left(\sum_{j=1}^2\|u_j^*\|_2^2\right)^p.
$$

Since we have already known by Theorem \ref{main} that the positive
solution of (\ref{GN1}) is uniquely determined by
$$
u^*_1=u^*_2=\left(\mu+\beta\right)^{-1/2p}\omega,
$$
we arrive at
\begin{align*}
\mathcal{J}({\bf
u^*})=\frac{(np)^{np/2}(2p+2-np)^{1-np/2}}{2(p+1)}\left(\frac{2}{(\mu+\beta)^{1/p}}\|\omega\|_2^2\right)^p.
\end{align*}
And therefore
\begin{align*}
\mathcal{K}_{n,p,\mu,\beta}=\frac{2(p+1)}{(np)^{np/2}(2p+2-np)^{1-np/2}\|\omega\|_2^{2p}}\cdot
\frac{(\mu+\beta)}{2^p}=\frac{(\mu+\beta)}{2^p}\mathcal{K}_{n,p},
\end{align*}
where the fact \cite{Weinstein1} that
$$
\mathcal{K}_{n,p}=\frac{2(p+1)}{(np)^{np/2}(2p+2-np)^{1-np/2}\|\omega\|_2^{2p}}
$$
is used, and this completes the proof of Corollary \ref{GN}.

\

{\bf Acknowledgement:} The first named author would like to thank
Prof. Congming Li for helpful discussions about uniqueness results
during his visit to Tsinghua University in July, 2007.

\end{document}